\documentclass[12pt,twoside]{amsart}
\usepackage{amsmath}
\usepackage{amssymb}
\usepackage{mathrsfs}         
\usepackage{srcltx}
\usepackage[curve]{xypic}
\usepackage{leftidx}		

\usepackage{tikz,tkz-euclide}
\usetikzlibrary{arrows,decorations.pathmorphing,decorations.markings,backgrounds,positioning,fit,petri,patterns,calc,intersections}
\usetkzobj{all} 

\usepackage[colorlinks=true, urlcolor=blue,bookmarks=true,bookmarksopen=true, citecolor=blue]{hyperref}
\usepackage{mathtools}        

\usepackage{fontawesome}

\numberwithin{equation}{section}

\newtheorem{theorem}{Theorem}[section]
\newtheorem{defn}[theorem]{Definition}

\newtheorem{corollary}[theorem]{Corollary}

\newtheorem{prop}[theorem]{Proposition}
\newtheorem{remark}[theorem]{Remark}

\def \begineq{\begin{equation}}
\def \endeq{\end{equation}}

\def \bb{\mathbb}

\def \CC{{\bb{C}}}

\def \QQ{{\bb{Q}}}
\def \RR{{\bb{R}}}
\def \R{{\bb{R}}}

\def \ZZ{{\bb{Z}}}
\def \Z{{\bb{Z}}}

\def \({\left(}
\def \){\right)}
\def \<{\langle}
\def \>{\rangle}

\def \C{{\mathbb{C}}}

\def \bar{\overline}
\def \deg{\mathrm{deg}}

\usepackage{graphicx}

\renewcommand{\tilde}{\widetilde}

\begin{document}

\title{The moduli space of real vector bundles of rank two over a real hyperelliptic curve}

 \author{Thomas John Baird, Shengda Hu}

\abstract
The Desale-Ramanan Theorem is an isomorphism between the moduli space of rank two vector bundles over complex hyperelliptic curve and the variety of linear subspaces in an intersection of two quadrics. We prove a real version of this theorem for the moduli space of real vector bundles over a real hyperelliptic curve. We then apply this result to study the topology of the moduli space, proving that it is relatively spin and identifying the diffeomorphism type for genus two curves. Our results lay the groundwork for future study of the quantum homology of these moduli spaces.
\endabstract

\maketitle


\section{Introduction}\label{sect:intro}

Given a Riemann surface $\Sigma$ of genus $g \geq 2$ and a line bundle $\xi \rightarrow \Sigma$ of odd degree, the moduli space $U_{\xi}$ of stable rank 2 vector bundles $E\rightarrow \Sigma$ with fixed determinant $\wedge^2 E \cong \xi$ is a non-singular projective variety of dimension $3g-3$ which is naturally endowed with a Kaehler metric.

Given an antiholomorphic involution $\tau: \Sigma \rightarrow \Sigma$, there is an induced antiholomorphic and antisymplectic involution on $U_{\xi}$. The fixed point set $U_{\xi}^\tau \subseteq U_{\xi}$ is a totally geodesic real Lagrangian submanifold of $U_{\xi}$. The points of $U_{\xi}^\tau$ correspond to holomorphic bundles over $\Sigma$ which admit a real structure \cite{BHH,S11}. The topology of $U_{\xi}^\tau$ has been investigated using Atiyah-Bott-Kirwan type methods: the mod 2 betti numbers were calculated in \cite{Baird13,Baird18,LS}, and when $g \geq 3$ the rational cohomology ring was calculated in \cite{Baird18}.

In the current paper, we study $U_\xi^{\tau}$ using an entirely different method in the special case when $\Sigma$ is a hyperelliptic curve. Recall that a hyperelliptic curve is a 2-fold ramified cover $\pi: \Sigma \rightarrow \C P^1$. If $\Sigma$ has genus $g \geq 2$, then $\pi$ has $2g+2$ ramification points  $W \subseteq \C P^1 $ also known as the Weierstrass points. Choose affine coordinates $\C P^1 = \C \cup \{\infty\}$ and let $\lambda_1,...,\lambda_{2g+2} \in \C$ denote the coordinates of $W$.  The following is due to Desale-Ramanan \cite{DesaleRamanan76} (see also \cite{NR, N} for the genus 2 case).

\begin{theorem}[Desale-Ramanan]\label{DRT}
If $\Sigma$ is hyperelliptic then $U_{\xi}$ is isomorphic to the subvariety of the Grassmannian $X \subseteq Gr_{g-1}( \C^{2g+2})$ consisting of $(g-2)$-linear subspaces contained in the intersection of the two quadrics $Z(q_0) \cap Z(q_1) \subset  \C P^{2g+1}$ where 
\begin{eqnarray*}
q_0 &=& x_1^2 + ...+ x_{2g+2}^2,  \\
q_1 &=& \lambda_1 x_1^2 +....+ \lambda_{2g+2} x_{2g+2}^2.
\end{eqnarray*}
\end{theorem}

Our main result (Theorem \ref{mainresult}) is a real version of the Desale-Ramanan Theorem. Namely, we prove that if $(\Sigma, \tau)$ is a real hyperelliptic curve, then $U_{\xi}^{\tau}$ is isomorphic to the subvariety of $Gr_{g-1}( \R^{2g+2})$ consisting of planes contained in the intersection of two explicit real quadrics. For genus $g=2$ curves, this result was proven by Shuguang Wang \cite{W}. 

The rest of the paper is devoted to applications of this result. In \S \ref{Grasssect} we produce formulas for the Stiefel-Whitney classes of  $U_{\xi}^{\tau}$ and use these to show that  $U_{\xi}^{\tau}$ is relatively spin as a Lagrangian submanifold of $U_{\xi}$. This implies that it has a well-defined quantum homology ring with integer coefficients \cite{BC}. 

In \S \ref{g2sect} we identify the diffeomorphism type of $U_{\xi}^{\tau}$ for all genus two examples. Previously, only the $\Z_2$-Betti numbers were known.

\section{Real hyper-elliptic curves}

A hyperelliptic curve, $\pi: \Sigma \rightarrow \C P^1$, is a 2-fold ramified cover over the complex projective line. We assume always that $\Sigma$ has genus $g \geq 2$. Such a curve admits an involution $\iota: \Sigma \rightarrow \Sigma$, which interchanges the two sheets of the cover and thus identifies $\Sigma/ \iota \cong \C P^1$. The ramification points $W \subset \C P^1$ are called the Weierstrass points.  Choosing affine coordinates on $\C P^1 = \C \cup \{\infty\}$ we can represent $W$ as a set of $2g+2$ points in the complex plane $\C$ (we assume always $\infty \not\in W$). 

A real structure on a Riemann surface $\Sigma$ is an anti-holomorphic involution $\tau$, that is a smooth map $\tau: \Sigma \rightarrow \Sigma$ which reverses the complex structure and satisfies $\tau^2 = Id_{\Sigma}$. The fixed point set of the involution, $\Sigma_\R := \Sigma^{\tau}$, is diffeomorphic to disjoint union of circles. In this paper, we consider real structures that are compatible with the hyperelliptic projection, i.e. \[\pi\circ \tau = \tau_{P^1} \circ \pi\]
with respect to the standard real involution $\tau_{P^1}$ on $\C P^1$, defined by $\tau_{P^1}([z:w]) := [\overline{z}:\overline{w}]$. A hyperelliptic curve admits a compatible real structure if and only if $\tau_{P^1}(W) = W$. If this happens, then $\Sigma$ admits a pair of compatible real structures $\tau$ and $\tau \circ \iota = \iota \circ \tau$, where $\iota$ is the hyperelliptic involution. For consistency, we will always choose $\tau$ so that $\infty \in \pi( \Sigma^{\tau})$ and $\infty \not\in \pi(\Sigma^{\tau\circ \iota})$. A typical situation for genus $g=2$ is illustrated in Figures \ref{Fig1} and \ref{Fig2} below.  The ramification points $W \subset \C \subset \C P^1$, drawn in orange, are arranged symmetrically about the real axis. The respective images $\pi( \Sigma^{\tau})$ and $\pi( \Sigma^{\tau \circ \iota})$ are drawn in green. 

\begin{figure}[htbp]
\begin{center}
\begin{tikzpicture}[scale=1,>=stealth] 
 \tkzInit[xmin=-3.5,xmax=3,ymin=-0.7,ymax=0.7]\tkzClip[space=0.1] 
 \begin{scope}
  \tkzDefPoints{-3/0/xl,3/0/xr,-2/0/p1,-1.2/0/p2, 0/0/p3, 1/0/p4, 1.8/0.7/q1, 1.8/-0.7/q2} 
  \tkzDrawSegments[color=green,<->,line width=1pt,opacity=0.8](xl,xr)
  \tkzDrawSegments[color=black,line width=1pt,opacity=1](p1,p2 p3,p4)
  \tkzDrawPoints[color=orange,fill=orange](p1,p2,p3,p4,q1,q2)
  \tkzLabelPoint[left](xl){$\RR$}
 \end{scope}
\end{tikzpicture}
\caption{$\tau$}
\label{Fig1}
\end{center}
\end{figure}

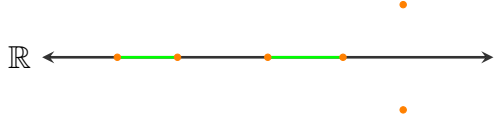
\begin{figure}[htbp]
\begin{center}
\begin{tikzpicture}[scale=1,>=stealth] 
 \tkzInit[xmin=-3.5,xmax=3,ymin=-0.7,ymax=0.7]\tkzClip[space=0.2] 
 \begin{scope}
  \tkzDefPoints{-3/0/xl,3/0/xr,-2/0/p1,-1.2/0/p2, 0/0/p3, 1/0/p4, 1.8/0.7/q1, 1.8/-0.7/q2} 
  \tkzDrawSegments[color=black,<->,line width=1pt,opacity=0.8](xl,xr)
  \tkzDrawSegments[color=green,line width=1pt,opacity=1](p1,p2 p3,p4)
  \tkzDrawPoints[color=orange,fill=orange](p1,p2,p3,p4,q1,q2)
  \tkzLabelPoint[left](xl){$\RR$}
 \end{scope}
\end{tikzpicture}
\caption{$ \iota \circ \tau$}
\label{Fig2}
\end{center}
\end{figure}

Given a real curve $(\Sigma, \tau)$ the fixed point set $\Sigma^{\tau}$ is a union of disjoint circles called the \emph{real circles}. The topological type of a real curve $(\Sigma, \tau)$ is determined by the genus of $g$, the number of path components of $\Sigma^\tau$ and whether or not $\Sigma \setminus \Sigma^{\tau}$ is connected (see \cite{GH}). The following proposition classifies real hyperelliptic curves topologically.

\begin{prop}\label{realcurvesprop}
Decompose $W = W_0 \cup W_+ \cup W_-$ , where $W_0$, $W_+$, $W_-$ are those Weierstrass points in $w \in \C$ whose imaginary part is respectively zero, positive, negative. Clearly $W_+$ and $W_-$ have equal cardinality so $W_0$ has even cardinality. Let $2n = \# W_0$. 

\begin{itemize}
\item[(i)] If $n>0$, then $\pi_0(\Sigma^{\tau}) = \pi_0(\Sigma^{\tau \circ \iota }) =  n$.  
\item[(ii)] If $n=0$ then $\pi_0(\Sigma^{\tau})= 1$ and  $\pi_0(\Sigma^{\tau \circ \iota }) =    0$.
\item[(iii)] If $1 \leq n \leq g$, then both $\Sigma \setminus \Sigma^{\tau}$ and $\Sigma \setminus \Sigma^{\tau \circ \iota}$ are connected.  If $n = g+1$ then both $\Sigma \setminus \Sigma^{\tau}$ and $\Sigma \setminus \Sigma^{\tau \circ \iota}$ are disconnected. If $n=0$, then  $\Sigma \setminus \Sigma^{\tau}$ is disconnected and $\Sigma \setminus \Sigma^{\tau \circ \iota}$ is connected.
\end{itemize}

\end{prop}

\begin{proof}
Statements (i) and (ii) are pretty clear.

For statement (iii),  observe that $\Sigma \setminus \Sigma^{\tau}$ is connected if and only if it is possible to draw a closed loop in $\C \setminus W$ which encloses an odd number of points in $W$ and crosses the real line. Then it is a simple matter of looking at pictures.
\end{proof}

\begin{remark}
Proposition \ref{realcurvesprop} implies that for genus $g \geq 4$, there exist topological types of real curves that cannot be realized as hyperelliptic curves.  Namely, those curves for which $1 < \pi_0( \Sigma^{\tau} ) < g+1$ and $\Sigma \setminus \Sigma^{\tau}$ is disconnected.
\end{remark}

\subsection{Real line bundles over real hyperelliptic curves}

Let $(\Sigma, \tau)$ be a real curve.  A real line bundle $\xi$ over $(\Sigma,\tau)$ is a line bundle that admits an antiholomorphic lift
$$ \xymatrix{ \xi \ar[r]^{\tilde{\tau}} \ar[d] &  \xi \ar[d] \\  \Sigma \ar[r]^\tau & \Sigma}$$ such that $\tilde{\tau}^2 = Id_{\xi}$. If $\tilde{\tau}$ exists then by Schur's Lemma it is uniquely determined up to multiplication by a unit scalar. The fixed point set $\xi^{\tilde{\tau}} \rightarrow \Sigma^{\tau}$ is an $\R^1$-bundle over $\Sigma^{\tau}$. The Stiefel-Whitney class $w_1(\xi^{\tilde{\tau}}) \in H^1(\Sigma^{\tau}; \Z_2)$ depends only on $(\Sigma, \tau, \xi)$, not on the choice of lift $\tilde{\tau}$. If $C \subset \Sigma^{\tau}$ is a real circle, then the following are equivalent
\begin{enumerate}
\item[(i)]  $\xi^{\tilde{\tau}}|_C$ is non-orientable (i.e. a Moebius band),
\item[(ii)]  $w_1(\xi^{\tilde{\tau}})(C) = 1$,
\end{enumerate}
We call $C$ odd with respect to $\xi$ if it satisfies these equivalent conditions. If $k$ is the number of odd circles with respect to $\xi$ then by (\cite{BHH} Prop. 4.1)
\begin{equation}\label{d=w=k}
\deg(\xi)  \equiv w_1(\xi^{\tilde{\tau}})(\Sigma^{\tau}) \equiv k~mod~2.
\end{equation}

The Desale-Ramanan Theorem requires the degree of $\xi$ to be odd, which implies that our real line bundle $\xi$ will always have an odd number of odd circles. In particular, we need only consider $(\Sigma, \tau)$ for which $\Sigma^{\tau}$ is non-empty. By (\cite{BHH} Prop. 4.1) if $\Sigma^{\tau} \neq \emptyset$, then real line bundles exist over $(\Sigma,\tau)$ in all degrees.  Tensoring by a real line bundle $L$ over $(\Sigma, \tau)$ determines an isomorphism $$ U_{\xi}^{\tau} \cong U_{\xi \otimes L^{\otimes 2}}^{\tau}$$ and $\deg(\xi \otimes L^{\otimes 2}) =  \deg(\xi) +2 \deg(L)$. Therefore in studying $U_{\xi}^{\tau}$ we may assume without loss of generality that $\xi$ has degree $2g+1$.

Given a real line bundle $\xi$ with real structure $\tilde{\tau}$, there always exists a $\tilde{\tau}$-invariant meromorphic section which determines a real divisor $D = \sum_{i} m_i p_i$ where $m_i \in \Z$ is the multiplicity of the zero of $s$ at $p_i \in \Sigma$. Observe that $\pi(D) = D$. A real circle in $C \subseteq \Sigma^{\tau}$ is odd with respect to $\xi$ if and only if $\sum_{p_i \in C} m_i$ is odd.

\section{The Desale-Ramanan Theorem}

We review what we need from Desale-Ramanan \cite{DesaleRamanan76}.

\subsection{The moduli space of bundles as a subvariety of a Grassmannian}\label{31}

Let $\pi: \Sigma \rightarrow \C P^1$ be a hyperelliptic curve of genus $g$ with hyperelliptic involution $\iota$. Let $W$ be the set of $2g+2$ Weierstrass points for $\Sigma$. We abuse notation and identify $W = \pi^{-1}(W) \subset \Sigma$.  Fix a line bundle $\xi \rightarrow \Sigma$ of degree $2g+1$.

Let $E \in U_{\xi}$ be a stable bundle with determinant isomorphic to $\xi$. There is a natural map $$H^0(\Sigma, E \otimes \iota^* E) \rightarrow  \bigoplus_{w \in W} E_w \otimes \iota^*E_w =  \bigoplus_{w \in W} E_w \otimes E_w$$
defined by restriction and identifying $\iota^*E_w = E_w$.  This map is $\iota$-equivariant. Restricting to the $-1$ eigenspaces, yields a map
\begin{equation}\label{grassmanninj}
 H^0(\Sigma, E \otimes \iota^* E) ^{-\iota} \rightarrow  \bigoplus_{w \in W} (E_w \otimes E_w)^{-\iota}  = \bigoplus_{w \in W} \wedge^2 E_w  \cong \bigoplus_{w \in W} \xi_w
 \end{equation}
where the last isomorphism uses $\wedge^2 E \cong \xi$, so is only natural up to a non-zero scalar in $\C^*$.  Desale-Ramanan prove that (\ref{grassmanninj}) is injective,  and $H^0(\Sigma, E \otimes \iota^* E) ^{-\iota}$ has fixed dimension $g+3$ for all $E$. We therefore obtain a morphism
$$ U_{\xi} \rightarrow Gr_{g+3} (V^*)$$
where $V^* = \oplus_{w \in W} \xi_w$. Applying duality, we can think of this as a morphism
$$ \phi:  U_{\xi} \rightarrow Gr_{g-1}(V) $$
where $V := \oplus_{w\in W} \xi_w^*$. Desale and Ramanan prove that $\phi$ is an embedding, so $U_{\xi}$ is isomorphic with the image $im(\phi)$. Moreover, they prove that $im(\phi)$ is equal to the set of $(g-2)$-linear subspaces in $P(V)$ that lie in the intersection of two quadrics $Q_0$ and $Q_1$. We construct these quadrics in the next section.

Now suppose we endow $\Sigma$ with an antiholomorphic involution $\tau$ which commutes with $\iota$. This means in particular that $\tau$ permutes the Weierstrass points.  If $\xi$ is a real line bundle over $(\Sigma,\tau),$ then $V = \oplus_{w \in W} \xi_w^*$ inherits an anti-linear involution inducing an involution on $Gr_{g-1}(V)$. It is clear from the construction that $\phi$ is $\tau$-equivariant.

\subsection{Line bundles and quadrics}\label{subsect:linebundleoncurve}

A pencil of quadratic forms on a vector space $V$ is determined by a surjective linear map $Q: S^2(V) \rightarrow U$ where $U$ is a vector space of dimension two.  In terms of a basis for $U$, this is defined by a linearly independent pair of quadratic forms $Q_0,Q_1: S^2(V) \rightarrow \C$. 

Given a line bundle $\xi$ of degree $2g+1$ over the hyperelliptic curve $\Sigma$, we want to construct a pencil of quadratic forms on $V$ where \begin{equation}\label{decomp}
V := \bigoplus_{w \in W} \xi_w^*.
\end{equation}
We do this by constructing a linear map $$Q:S^2(V) \rightarrow H^0( h^{-2g-1}(W))$$ where $h$ is the hyperplane bundle over $\C P^1$ and $W$ is the Weierstrass divisor. Our pencil will be diagonal with respect to the decomposition (\ref{decomp}) so it is determined by morphisms
$$ q_w: (\xi_w^*)^{\otimes 2} \rightarrow H^0( h^{-2g-1}(W))$$ 
for each $w \in W$. We define $q_w$ to be the composition of morphisms (\ref{eqa}), (\ref{eqb}), (\ref{eqc}) defined below. Fix an isomorphism
\begin{equation}\label{xih}
\phi:  \xi^* \otimes \iota^* \xi^* \cong  \pi^*(h)^{-2g-1}.
 \end{equation}
At a Weierstrass point $w \in W$, $\phi$ restricts to an isomorphism
\begin{equation}\label{eqa}
(\xi_w^*)^{\otimes 2} = \xi^*_w \otimes \iota^* \xi^*_w \cong   \pi^*(h)^{-2g-1}_w = h_w^{-2g-1}
\end{equation}
(recall we abuse notation identifying $w$ and $\pi(w)$). Evaluation at $w$ determines an isomorphism 
\begin{equation}\label{eqb}
h_w^{-2g-1} \cong H^0( h^{-2g-1}(W - w)),
\end{equation}
and we have the natural inclusion 
\begin{equation}\label{eqc}
H^0( h^{-2g-1}(W - w)) \hookrightarrow H^0( h^{-2g-1}(W)).
\end{equation}

Choose affine coordinates on $\C P^1 = \C \cup \{\infty\}$ with variable $t$ and Weierstrass points $\{t_w \in \C\}$. In these coordinates we have $$ H^0( h^{-2g-1}(W)) = \{ (at +b) R(t)| a,b\in \C \} $$ 
where $R(t) = t^{2g+1} \prod_{w \in W} (t-t_w)^{-1}$. The quadratic forms $Q_0, Q_1$ are defined by the identity $Q = (Q_0t  - Q_1) R(t)$.

Let $u $ be a meromorphic section of $\xi$ which may be chosen to have no poles or zeros on $W$ so that $\{u_w|w \in W\}$ is a basis for $V$. Let $D$ be the divisor of $u$ and express  $$\pi(D) =  \sum m_i \alpha_i$$ in affine coordinates $\alpha_i \in \C \subset \C P^1$  where $\sum m_i = -2g-1$. Then up to a scalar (which we may fix to equal one),  $u \otimes \iota^*(u)$ is sent by $\phi$ to the pull-back of $\prod_i (t - \alpha_i)^{m_i}$.  Tracing through the definition of $q_w$ we get
$$  q_w( u_w^2) =  (t -t_w) \prod_i (t_w - \alpha_i)^{m_i}  \prod_{w' \neq w} (t_w -t_{w'}) R(t).$$
Therefore

\begin{eqnarray}
Q_0(u_w^2) &:=&  \prod_i (t_w - \alpha_i)^{m_i}  \prod_{w' \neq w} (t_w -t_{w'}) \label{Q0}\\
Q_1(u_w^2) &:=& t_w Q_0(u_w^2).
\end{eqnarray}
We recover the expression in Theorem \ref{DRT} simply by replacing the basis $u_w$ with $\frac{1}{ \sqrt{Q_0(u_w^2)}} u_w$. To get the real version, we must choose a different basis.

\section{Main Result}\label{sect:mainresult}


Suppose now that $\pi: \Sigma \rightarrow \C P^1$ is a real hyperelliptic curve with anti-holomorphic involution $\tau$ compatible with the standard involution $\tau_{P^1}$ on $\C P^1$. Suppose this lifts to a real structure  $\tilde{\tau}$ on $\xi$ and choose a $\tilde{\tau}$-invariant meromorphic section $u$. We want to replace the basis $\{u_w| w \in W\}$ of $V$ by a basis whose elements are fixed by $\tau$ and with respect to which the expressions for $Q_0$ and $Q_1$ simplify. 

Given affine coordinates $\C \cup \{ \infty\}$, recall that if $\# W_0 = 2n >0$, then $\pi(\Sigma^{\tau}) \cap \C$ is a union of $n$ disjoint closed intervals in $\R$, two of which are half infinite. Declare such an interval $I \subseteq \R$ to be odd with respect to $\pi(D) = \sum_i m_i \alpha_i$ if $\sum_{\alpha_i \in I} m_i$ is odd. Note that a bounded interval is odd if and only if it is the image of an odd circle for $\xi$. To simplify what follows we require when $n>0$ that the negative half infinite interval be even. This can always be arranged by changing the affine coordinate chart using a Moebius transformation. 

\begin{theorem}\label{mainresult}
Suppose that $W_0$ consists of real numbers $r_1<r_2< ... < r_{2n}$  while $W_+$ consists of complex numbers $a_1 + i b_1, ...., a_{g+1-n} +ib_{g+1-n}$ where $b_i>0$.  Assume that $ \infty \in \pi( \Sigma^{\tau})$. There is a choice of coordinates $x_1...,x_{2r},, z_1,...,z_s, w_1,...,w_s$ on $V$, in which the quadratic forms $Q_0$ and $Q_1$ are expressed as polynomials 
\begin{eqnarray}
q_0 &:=&  \sum_{i=1}^{2n} \epsilon_i \left(x_{2i-1}^2-x_{2i}^2\right) + \sum_{j = 1}^{s} \left(z_{j}^2 -w_{j}^2\right)\label{eqn4.1}  \\
q_1 &:=&  \sum_{i=1}^{2n} \epsilon_i \left(r_{2i-1} x_{2i-1}^2-r_{2i}x_{2i}^2\right)   +  \sum_{j=1}^{s} \left( a_{j} z_{j}^2 - a_j z_{j}^2   +  2b_j z_{j}w_{j} \right) \label{eqn4.2}
\end{eqnarray}
where $\epsilon_i \in \{\pm 1\}$ equals $(-1)^{(N_i+1)}$ where $N_i$ is the number of odd intervals to the right of $r_{2i-1}$. This convention implies $\epsilon_1 =1$.
\end{theorem}
\begin{proof}
Clearly $\tau_{P^1}$ preserves both $\pi(D)$ and $W$. Let $\nu$ be the permutation of $W$ such that $\tau_{P^1}(w) =\nu(w)$. The induced anti-linear involution on $V = \oplus_{w \in W} \xi_w^*$ sends a linear combination $ \sum_{w \in W} \lambda_w u_w$ to $\sum_{w \in W}\overline{\lambda}_w u_{\nu(w)}$.

Introduce the notation $ R_w e^{i \theta_w} := Q_0(u_w^2)$ where $R_w > 0$ and $e^{i\theta_w} \in [0, 2\pi)$. This is well-defined because from (\ref{Q0}) we see $Q_0(u_w^2) \neq 0$. Define the basis 
 $$ B := \{ v_w ~|~w \in W_0\} \cup \{ v'_w, v''_w~|~w \in W_+\}  $$
where
\begin{eqnarray*}
 v_w &:= &  \frac{1}{\sqrt{ R_w}} u_w \\
v'_w &:=&   \frac{e^{-i \theta_w/2}}{\sqrt{ 2 R_w}}u_w +  \frac{e^{i \theta_w/2}}{\sqrt{2 R_w}} u_{\nu(w)}\\
v''_w &:=&  -i\Big(  \frac{e^{-i \theta_w/2}}{\sqrt{ 2 R_w}}u_w -  \frac{e^{i \theta_w/2}}{\sqrt{2 R_w}} u_{\nu(w)} \Big) .
\end{eqnarray*}
Clearly each basis vector in $B$ is fixed by conjugation. It is straightforward to check that $Q_0$ is diagonal in this basis, with
\begin{eqnarray*}
Q_0(v_{w}^2) &=& e^{i \theta_w}\\
Q_0((v_w')^2) &=&1 \\
Q_0((v_w')^2) &=& -1 \\
\end{eqnarray*}
while $Q_1$ satisfies
\begin{eqnarray*}
Q_1(v_{w}^2) &=& t_w e^{i \theta_w}\\
Q_1((v_w')^2) &=& (t_w + \bar{t}_w)/2 = Re(t_w) \\
Q_1((v_w')^2) &= &- (t_w + \bar{t}_w)/2 = - Re(t_w) \\
Q_1(v_w' v_w'') &= &- i (t_w - \bar{t}_w)/2 = Im(t_w) \\
\end{eqnarray*}
with all other entries zero. Finally, check that if $w \in W_0$, then $e^{i \theta_w} = (-1)^N$ where $$ N = \#\{w' \in W_0 | t_{w'} > t_w\} +  \sum_{\alpha_i \in \R, \alpha_i > t_w} m_i.$$ Therefore $e^{i \theta_w} = \epsilon_i$ if $t_w = r_{2i-1}$ and $e^{i \theta_w} = -\epsilon_i$ if $t_w = r_{2i}$ as desired.
\end{proof}

\section{Stiefel-Whitney classes}\label{Grasssect}

Let $G_{\C} := Gr_{g-1} (\CC^{2g+2})$ with tautological bundle $V_{\C} \rightarrow G_{\C}$ and let $G := Gr_{g-1}(\R^{2g+2})$ and $V$ be their real counterparts. We saw in \S \ref{31} that there is a $\tau$-equivariant embedding $\phi: U_{\xi} \hookrightarrow G_{\C}$, which restricts to an embedding of $U_{\xi}^{\tau}$ into $G$. The image of $\phi$ coincides with the set of $(g-1)$-planes on which the quadratic forms $Q_0, Q_1$ both vanish.

\begin{prop}
The image of $\phi$ equals the zero locus $Z(s)$ of a $\tau$-invariant section $s \in H^0(G; S^2(V^*_{\C})^{\oplus 2})$ which intersects the zero section transversely. Therefore $U^\tau_{\xi} \cong Z(s)^\tau = Z(s^{\tau})$ where $s^{\tau}$ is restricted section of $S^2(V^*)$. 
\end{prop}

\begin{proof}
The identification of $Z(s)$ with linear subspaces of $Q_0, Q_1$ is explained in Borcea \cite{Borcea}. Transversality follows from (\cite{Borcea} Corollary 2.2) because the pencil of quadrics $xQ_0+yQ_1$ is generic. Invariance under $\tau$ is clear.
\end{proof}

\begin{corollary}
Let $(\Sigma_0, \tau_0)$ and $(\Sigma_1, \tau_1)$ be real hyperelliptic curves of the same genus $g \geq 2$ equipped with real line bundles $\xi_1$ and $\xi_2$ respectively of degree $2g+1$. Then
 $U_{\xi_0}^{\tau_0}$ and $U_{\xi_1}^{\tau_1}$ are cobordant. 
\end{corollary}

\begin{proof}
There exist real sections $s_0,s_1 \in \Gamma(S^2(V^*))$ such that $U_{\xi_i}^{\tau_i}= Z(s_i) := s^{-1}_i(0)$.  Choose a homotopy $s: Gr_{g-1}(\R^{2g+2}) \times I  \rightarrow  S^2(V^*)$ from $s_0$ to $s_1$ which intersects the zero section transversally. Then $Z(s)$ provides a cobordism between $Z(s_0) = M_0$ and $Z(s_1) = M_1$.  
\end{proof}

%

Let $N := im(\phi)^{\tau} \cong U_{\xi}^{\tau}$. We have the following isomorphism of vector bundles
\begin{equation}\label{TGTM}
TG|_N \cong TN \oplus S^2(V^*|_N)\oplus S^2(V^*|_N)
\end{equation}
which can be used to compute Stiefel-Whitney classes.  

\begin{prop}\label{chernform}\label{PropSW}
Let $N \cong U_\xi^\tau$ be as above. The total Stiefel-Whitney class of the tangent bundle of $N$ is equals
\begin{equation}\label{SWeq}
w(TN) =  w(V^*|_N)^{2g+2}  w(V|_N \otimes V^*|_N)^{-1} w( S^2(V^*)|_N)^{-2}.
\end{equation}
In particular, we have
\begin{eqnarray*}
w_1(U_\xi^\tau) &=& 0\\
w_2(U_\xi^\tau) &=&  (g+1) \phi^*(w_1)^2
\end{eqnarray*}
where $w_i$ is the tautological Stiefel-Whitney class in $H^i(G;\Z_2).$ 
\end{prop}

\begin{proof}
By (\ref{TGTM}) and the Whitney sum formula, we get 
$$ w(TN) = w(TG|_N) w( S^2(V^*)|_N)^{-2} .$$
We have the well-known isomorphism $ TG \cong Hom(V, W)$ where $V$ and $W$ are the tautological bundle and its orthogonal complement respectively. In particular, $V \oplus W \cong G \times \R^{2g+2}$.  Therefore
$$  TG \oplus Hom(V, V) = Hom( V, \R^{2g+2}) = V^* \oplus ... \oplus V^*,$$
so by the Whitney sum formula
\begin{eqnarray*}
w(TG)  &  = & w(V^*)^{2g+2}  w(V \otimes V^*)^{-1}
\end{eqnarray*}
proving (\ref{SWeq}). By definition 
\begin{eqnarray*}
 w(V) = w(V^*) & =& 1+w_1+w_2+....+w_{g-1}.\\
\end{eqnarray*}
A simple calculation using the splitting principle gives
\begin{eqnarray*}
  w(V \otimes V^*) &=&  1 + (rk(V)-1)w_1^2 + O(3)\\
                              &=& 1 + g w_1^2  + O(3)\\
w( S^2(V^*))^2 &=&  1+ w_1(S^2(V^*))^2 +O(3)\\ 
                        &=&1 + gw_1^2 + O(3)\\
\end{eqnarray*}
where $O(3)$ is a sum of terms in degree three or higher. The values of $w_i(U_{\xi}^{\tau})$ follow by direct calculation and functoriality.
\end{proof}

Before stating the next corollary we to introduce some terminology.  

\begin{defn}
Let $L$  be a Lagrangian submanifold in a symplectic manifold $(M,\omega)$.  We call $L$
\begin{itemize}
\item relatively pin if $w_2(L)$ lies in the image of $H^2(M;\ZZ_2) \rightarrow H^2(L;\ZZ_2)$. 
\item relatively spin if it is relatively pin and orientable. 
\item spin if it is orientable and $w_2(L) =0$. 
\end{itemize}
\end{defn}

\begin{corollary}\label{pinspin}
The Lagrangian submanifold $U_{\xi}^{\tau} \subset U_{\xi}$ is relatively spin for all $g \geq 2$ and is spin if  $g$ is odd.
\end{corollary}

\begin{proof}
That $U_{\xi}^{\tau}$ is orientable follows from $w_1(U_{\xi}^{\tau})=0$ which is proven in Proposition \ref{PropSW}.
Consider the commutative diagram of inclusions
$$ \xymatrix{   U_{\xi}  \ar[r]^{\phi} & G_\C \\   U_{\xi}^{\tau} \ar[u] \ar[r]^{\phi}   &  G \ar[u]^i }. $$
We have $w_2( U_{\xi}^{\tau}) = (g+1) \phi^*(w_1)^2$ by Proposition \ref{PropSW},  and we have $w_1^2 = i^*(u)$ where $u$ is the pull-back of the generator of $H^2(G_{\C};\ZZ_2)$ (\cite{MS} problem 15-A). Commutativity completes the argument.
\end{proof}

\begin{remark}
\rm{
 We note that the moduli space $U_\xi$ is monotone with minimal Chern number $2$, which implies that $U_\xi^\tau$ is monotone with minimal Maslov number greater than or equal to two (see \cite{Baird18} Theorem 1.6). It follows from \cite{BC} 
 that the quantum homology of the Lagrangian submanifold $U_\xi^\tau$ is well defined over the Novikov ring with integer coefficients. It is well-known that Dehn twist on $\Sigma$ induces fibered Dehn twist of $U_\xi$ (e.g. \cite{WehrheimWoodward} and references therein). For the real curve $(\Sigma, \tau)$, we can consider a real Dehn twist. In this case, the corresponding moduli spaces of real bundles can be related by a Lagrangian cobordism in Lefschetz fibration (c.f. Biran-Cornea \cite{BC15}). We hope to pursue this in future work.
 }
\end{remark}

\section{The genus 2 case}\label{g2sect}

Given a generic intersection of real quadrics $X = Z(q_0) \cap Z(q_1) \subseteq \RR P^N$ we can form the double cover $\tilde{X} \rightarrow X$ by pulling back the double cover $S^{N} \rightarrow \RR P^{N}$. The diffeomorphism types of such $\tilde{X}$ were classified by Guti\'errez and  L\'opez de Medrano \cite{GL}, which we review below. 

Suppose $q_0$ and $q_1$ are determined by symmetric real matrices $A$ and $B$. Up to a small perturbation that doesn't affect the diffeomorphism type, we may assume that $A^{-1}B$ is diagonalizable with distinct eigenvalues. Introduce real variables $x_1,...,x_r, u_1,v_1,...,u_s,v_s$ where each $x_i$ corresponds to the eigenspace for a real eigenvalue of $A^{-1}B$ and each pair $u_i, v_i$ corresponds to the real part of the sum of eigenspaces for a complex conjugate pair of eigenvalues. Then the coefficients of $q_0, q_1$ can be continuously varied without changing the diffeomorphism type to a pair of quadratic forms
\begin{eqnarray}
p_0 &=& \sum_{i=1}^ra_i x_i^2 + \sum_{j=1}^s( u_j^2 -v_j^2) \label{eq6.1} \\
 p_1 &=& \sum_{i=1}^r b_i x_i^2 +\sum_{j=1}^s 2u_jv_j \label{eq6.2}
\end{eqnarray}
where the $b_i/a_i$ are the real eigenvalues of $A^{-1}B$.

Consider the set of points $\lambda_i = (a_i, b_i) \in \R^2$.  The fact that the intersection is generic implies that $(0,0)$ does not lie on the line segment joining $\lambda_i$ and $\lambda_j$ for any pair $i,j \in \{1,...,r\}$. If the coefficients $\lambda_i$ are continuously varied without violating this property, then the diffeomorphism type of the intersection doesn't change. The $\lambda_i$ can in this way be put into a standard form, such that all of the $\lambda_i$ lie on $2l+1$ roots of unity (in $\R^2 = \C$) for a minimal value $l$. This determines a cyclically ordered partition  $ r = n_1 +...+n_{2l+1}$ where $n_i>0$ counts the multiplicity of $\lambda_i$ at the $i$th root of unity. The diffeomorphism type of $\tilde{X},$ hence also $X,$ is determined by $s$ and the partition $r = n_1+...,+n_{2l+1}$.

\begin{prop}
Suppose that $(\Sigma, \tau)$ is a real hyperelliptic curve of genus $g=2$ with $2n$ real Weierstrass points and let $\xi$ be a real line bundle with $k$ many odd circles. Then the quadric intersection type of $Z(q_0) \cap Z(q_1)$ and the diffeomorphism type of the double cover $\tilde{U}_{\xi}^{\tau}$ are as follows.

\bigskip
\begin{tabular}{|c|l|c|c|c|c|}
	\hline
(n,k)& quadric intersection type &  diffeomorphism type of  $\tilde{U}_{\xi}^{\tau}$  \\
\hline
(0,1) & s=3, r=0 &$\RR P^3$  \\
(1,1) & s=2, r=2 &$S^1 \times S^2$   \\
(2,1) & s=1, r=1+1+2 &$\#_3 (S^1 \times S^2)$ \\
(3,1) & s=0, r=1+1+1+1+2 &$\#_5 (S^1 \times S^2)$  \\
(3,3) & s=0, r=2+2+2 &$T^3$  \\
\hline	
\end{tabular}
\bigskip

\end{prop}

\begin{proof}
Comparing to formulas (\ref{eqn4.1}), (\ref{eqn4.2}) with (\ref{eq6.1}), (\ref{eq6.2}), we see that  $r = 2n$ is equal to the number of real Weierstrass points and $2s$ is the number of non-real Weierstrass points. The partitions can be worked out by hand case-by-case. The diffeomorphism types follows from the main theorem of \cite{GL}.
\end{proof}

Guti\'errez and  L\'opez de Medrano do not provide a general formula for the diffeomorphism type of the intersection of projective quadrics $X$ itself, but it is not hard to determine it for our examples.

\begin{theorem}\label{diffeotype}
Suppose that $(\Sigma, \tau)$ is a real hyperelliptic curve of genus $2$ with $2n$ real Weierstrass points and let $\xi$ be a real line bundle with $k$ many odd circles (as for Theorem 6.1). The diffeomorphism type of $U_{\xi}^{\tau}$ is
\bigskip

\begin{tabular}{|c|c|c|c|c|c|}
	\hline
(n,k)&  diffeomorphism type of $U_{\xi}^{\tau}$  \\
\hline
(0,1) & $L(4,1)$  \\
(1,1) & $S^1 \times S^2$   \\
(2,1) & $\#_2 (S^1 \times S^2)$ \\
(3,1) & $\#_3 (S^1 \times S^2)$  \\
(3,3) & $T^3$  \\
\hline	
\end{tabular}
\bigskip

\end{theorem}

\begin{proof}
For the case $(0,1)$ we know that $U_{\xi}^{\tau}$ is diffeomorphic to the projective quadric defined by 
\[\sum_{i=1}^3 \left( u_i^2-v_i^2 \right) = \sum_{i=1}^3 2u_i v_i = 0\]
If we regard these as affine equations and impose the extra affine condition 
\[\sum_{i=1}^3 \left(u_i^2 +v_i^2\right)=1\]
then these define the unit tangent bundle of $S^2$. Taking the projective quotient gives the unit tangent bundle of $\R P^2$ which is diffeomorphic to the lens space $L(4,1)$ (see \cite{K}).

The cases (1,1), (2,1), and (3,1) are each diffeomorphic to an intersection of quadrics of the form
\begin{eqnarray*} 
x_1^2 +x_2^2 + F(x_3,x_4,x_5,x_6) &=& 0 \\ 
0 + F(x_3,x_4,x_5,x_6) &=& 0 
\end{eqnarray*}
which admits an $SO(2)$-action defined by rotating the coordinates $x_1,x_2$. Note that points in the intersection must either be fixed by this rotation (if $x_1=x_2=0$) or have trivial stabilizer (because the first equation implies there are no non-zero solutions of the form $(a_1,a_2,0,0,0,0)$). In all three cases it is easy to verify that the fixed point set is non-empty. Since by Corollary \ref{pinspin} we also know that $U_{\xi}^{\tau}$ is orientable, a result of Raymond (\cite{R} Theorem 1) implies that $U_{\xi}^{\tau}$ is diffeomorphic to a connected sum of copies of $S^1 \times S^2$. The mod 2 Betti numbers of $U_{\xi}^{\tau}$ were calculated in \cite{Baird18}, from which we can determine the number of copies of $S^1 \times S^2$ occurring in each case.

The case $(3,3)$ actually admits an effective action by the 3-torus $SO(2)^3$ (see \cite{GL} \S 4.1) and therefore must be diffeomorphic to a 3-torus.
\end{proof}

\begin{remark}
\rm{
 The case $(3,1)$ was considered by Saveliev and Wang \cite{SW}. They correctly proved $U_\xi^{\tau}$ has rational Poincar\'e polynomial equal to $1+3t+3t^2+t^3$, but is not homeomorphic to $T^3$.
}
\end{remark}

\begin{remark}
\rm{
 In the proof of Theorem \ref{diffeotype}, we made use of the existence of torus actions on (a manifold diffeomorphic to) $U_{\xi}^{\tau}$. These circle actions may be related to the Jeffrey-Weitsman torus action \cite{JW} obtained from Goldman's integral system \cite{G}. The Narasimhan-Seshadri Theorem determines a diffeomorphism between $U_{\xi}$ and the twisted representation variety $M = Hom_{-1}(\pi_1(\Sigma), SU(2))/SU(2)$. Jeffrey and Weitsman produced Hamiltonian $SO(2)$ actions defined on dense open subsets of $M$ corresponding to embedded circles in the Riemann surface $\Sigma$. These $SO(2)$ actions commute whenever the embedded circles are disjoint. The action is not defined everywhere on $M$ because the Hamiltonian functions are not everywhere differentiable.  However, in unpublished work by the first author it is shown that for an embedded circle coinciding with an odd circle for a real curve $(\Sigma,\tau)$, the corresponding $SO(2)$-action restricts to a globally defined circle action on the submanifold $M^{\tau}$ identified with $U_\xi^{\tau}$. Therefore $M^{\tau}$ admits a torus action of rank equal to the number of odd circles. This may correspond to the torus actions employed above.
}
\end{remark}

\textbf{Acknowledgements:} The authors would like to thank Joel Kamnitzer, Luis Haug and Francois Charette for insightful discussions. Thanks also to Guti\'{e}rrez and Lopez de Medrano for answering our questions about their work. This research was supported in part by NSERC Discovery grants.

\end{document}